\documentclass[12pt]{amsart}

\usepackage{amssymb}
\usepackage{amsthm}
\usepackage{mathtools}
\usepackage{tikz-cd}
\usepackage{color}
\usepackage{upref}
\usepackage{enumitem}
\usepackage{verbatim}
\usepackage[normalem]{ulem}
\usepackage{hyperref}

\allowdisplaybreaks

\newtheorem{thm}{Theorem}[section]
\newtheorem{prop}[thm]{Proposition}

\newtheorem{lem}[thm]{Lemma}

\theoremstyle{definition}

\newtheorem{rem}[thm]{Remark}
\newtheorem{ex}[thm]{Example}

\numberwithin{equation}{section}

\newcommand{\secref}[1]{Section~\textup{\ref{#1}}}

\newcommand{\thmref}[1]{Theorem~\textup{\ref{#1}}}

\newcommand{\lemref}[1]{Lemma~\textup{\ref{#1}}}
\newcommand{\propref}[1]{Proposition~\textup{\ref{#1}}}

\newcommand{\remref}[1]{Remark~\textup{\ref{#1}}}
\newcommand{\exref}[1]{Example~\textup{\ref{#1}}}

\newcommand{\bb}[1]{\mathbb{#1}}
\newcommand{\cc}[1]{\mathcal{#1}}

\newcommand{\C}{\bb C}

\newcommand{\Z}{\bb Z}
\newcommand{\T}{\bb T}

\newcommand{\KK}{\ensuremath{\cc K}}

\renewcommand{\AA}{\cc A}

\newcommand{\Chi}{\raisebox{2pt}{\ensuremath{\chi}}}

\newcommand{\minus}{\setminus}
\renewcommand{\epsilon}{\varepsilon}
\newcommand{\<}{\langle}
\renewcommand{\>}{\rangle}

\newcommand{\inv}{^{-1}}

\newcommand{\what}{\widehat}

\newcommand{\ann}{^\perp}
\newcommand{\id}{\text{id}}

\renewcommand{\subset}{\subseteq}

\renewcommand{\bar}{\overline}

\newcommand{\variso}{\overset{\simeq}{\longrightarrow}}
\newcommand{\xt}{\otimes}

\newcommand{\mtx}[1]{\begin{pmatrix} #1 \end{pmatrix}}

\newcommand{\cst}{\ensuremath{C^*}}
\newcommand{\csta}{\ensuremath{C^*}-algebra}
\newcommand{\cstg}{\ensuremath{C^*(G)}}

\newcommand{\cog}{\ensuremath{C_0(G)}}
\newcommand{\cg}{\ensuremath{C(G)}}
\newcommand{\ag}{\ensuremath{A(G)}}
\newcommand{\ltwo}{\ensuremath{L^2(G)}}
\newcommand{\dg}{\ensuremath{\delta_G}}

\newcommand{\cstwg}{\ensuremath{C^*_\omega(G)}}

\newcommand{\mn}{\ensuremath{M_n}}

\newcommand{\fixed}{\ensuremath{A^\delta}}
\newcommand{\rep}{representation}
\newcommand{\hm}{homomorphism}

\newcommand{\fb}{Fell bundle}

\newcommand{\spn}{\operatorname{span}}

\newcommand{\ad}{\operatorname{Ad}}

\newcommand{\aut}{\operatorname{Aut}}

\newcommand{\diag}{\operatorname{diag}}
\newcommand{\su}{\operatorname{SU}(2)}
\newcommand{\so}{\operatorname{SO}(3)}
\newcommand{\un}{\operatorname{U}(n)}
\newcommand{\pun}{\operatorname{PU}(n)}
\newcommand{\spec}{\operatorname{sp} \delta}

\newcommand{\midtext}[1]{\quad\text{#1}\quad}
\newcommand{\righttext}[1]{\quad\text{#1 }}

\begin{document}

\title[Coactions of compact groups on $M_n$]{Coactions of compact groups on $M_n$}
\author[Kaliszewski]{S. Kaliszewski}
\address{School of Mathematical and Statistical Sciences, Arizona State University, Tempe, AZ 85287}
\email{kaliszewski@asu.edu}
\author[Landstad]{Magnus~B. Landstad}
\address{Department of Mathematical Sciences\\
Norwegian University of Science and Technology\\
NO-7491 Trondheim, Norway}
\email{magnus.landstad@ntnu.no}
\author[Quigg]{John Quigg}
\address{School of Mathematical and Statistical Sciences, Arizona State University, Tempe, AZ 85287}
\email{quigg@asu.edu}

\dedicatory{We dedicate this paper to the memory of our friend and colleague Iain Raeburn.}

\date{\today}

\begin{abstract}
We prove that
every coaction of a compact group on a finite-dimensional \csta\ is associated with a \fb.
Every coaction of a compact group on a matrix algebra is implemented by a unitary operator.
A coaction of a compact group on $M_n$ is inner if and only if
its fixed-point algebra has an abelian \cst-subalgebra of dimension $n$.
Investigating the existence of effective ergodic coactions on $M_n$ reveals that $\so$ has them, while $\su$ does not.
We give explicit examples of the two smallest finite nonabelian groups having 
effective ergodic coactions on $M_n$.
\end{abstract}

\maketitle

\section{Introduction}\label{s:intro}

Coactions of locally compact groups on \csta s were invented to generalize crossed-product duality to nonabelian groups.
Coactions of abelian groups correspond to actions of the dual groups,
via Fourier transform,
and the theories correspond remarkably closely.
In particular,
coactions of discrete groups correspond to actions of compact groups,
and
coactions of compact groups correspond to actions of discrete groups.
Since actions of compact groups are the easiest, it is not surprising that coactions of discrete groups are also easiest; in fact, they are essentially the same as Fell bundles (over discrete groups).
On the other hand, actions of discrete groups can be quite hard, and correspondingly coactions of compact groups can be expected to be hard too.
So, we felt that it is a good time to look closely at coactions of compact groups.
As a first step, we compensate by restricting the \csta\ to be finite-dimensional --- in fact, as our title indicates, we concentrate on \mn.
A long time ago Iain asked one of us exactly this question, which we now can answer.

As a reward, we show that coactions of compact groups on finite-dimensional \csta s are given by Fell bundles (see \thmref{fb}).
We remark (see \remref{general}) that in fact the proof would work for arbitrary amenable groups (amenability would be needed to ensure that the Fourier algebra $A(G)$ separates points of $M(\cstg)$),
but we do not go into details here since we did not want to break up the flow with compact groups.
By the way, a Fell bundle over a locally compact group is required to have a compatible topology. However, every Fell bundle we will encounter over our compact group $G$ will have only finitely many nonzero fibres, so the topology of $G$ becomes irrelevant.

Next, since actions of discrete groups on \mn\ are unitarily implemented, it is natural to expect something along these lines for coactions of compact groups, and we prove this in \thmref{Mn implement}
(see also \cite[Theorem~3]{wassermann}).
The key here is to note that the passage from actions of an abelian group $G$ to coactions of its dual group $\what G$ involve passing from \cog\ to \cstg, and that the individual automorphisms of $G$ can be recovered from \cog\ by composing with irreducible \rep s.

Inner coactions of $G$ are determined by \hm s of \cg, and we use this to characterize the inner coactions on \mn\ as 
those
whose fixed-point algebras contain a maximal abelian subalgebra of \mn\ (see \thmref{fixed inner}).
We illustrate our techniques by considering inner coactions on $M_2$ and $M_3$ (Examples~\ref{m2} and \ref{m3}).

In some sense at the opposite extreme from inner coactions are the ergodic ones, and we might as well assume that our ergodic coactions are effective (that is, every spectral subspace $A_x$ is nonzero), and this forces us down to finite groups --- indeed, those of order $n^2$ since our coactions are on \mn.
In fact, the existence of an effective ergodic coaction of $G$ on \mn\
is characterized by a result of Kleppner \cite{kleppner} regarding ``regular elements''.
The case of $M_p$ with $p$ prime is unsurprisingly quite special, forcing $G=\Z_p\times \Z_p$ (see \thmref{prime}).
As an application of our techniques, we show that $\su$ has no ergodic coaction on any \mn\ (with $n>1$), while $\so$ does (see \exref{su}).
In \exref{explicit} we give an explicit example of two finite nonabelian groups having 
effective ergodic coactions on $M_n$.
The groups have order 16, and this is minimal.

In \secref{s:s3} we examine coactions on \mn\ of the smallest nonabelian (discrete) group $S_3$,
and in particular we characterize the effective inner coactions, which require $n\ge 4$ (see \propref{s3 inner}).
Finally, in \secref{sec:end} we briefly indicate further possible research (and we comment ``Clearly, we have only scratched the surface'').

This research is part of the EU Staff Exchange project 101086394 ``Operator Algebras That One Can See''. It was partially supported by the University of Warsaw Thematic Research Programme ``Quantum Symmetries''. We thank our host Piotr M. Hajac and IMPAN for the hospitality.

\section{Preliminaries}\label{s:prelim}

Throughout, $G$ is a compact group with left Haar measure, and $A,B,\dots$ typically denote \csta s.
We refer to \cite[Appendix~A]{enchilada} for actions, and coactions.
If $\delta:A\to M(A\xt\cstg)$ is a coaction of $G$ on $A$, we say that the pair $(A,\delta)$ is a coaction.

Since $G$ is compact, all coactions are normal (and maximal).
Furthermore,
the Fourier algebra \ag\ coincides with the Fourier-Stieltjes algebra $B(G)$,
so it is a unital Banach *-algebra\footnote{with norm from the dual space $\cstg^*$, pointwise multiplication, and involution given by $f^*(x)=\bar{f(x\inv)}$} that is dense in \cg.
Most of the background on \ag\ is contained in \cite{eymard}.
Most importantly
\cite[Th\'eor\`eme~3.34]{eymard} says that the dual (ie., the maximal ideal space) of \ag\ can be identified with $G$ via
\[
\Chi_x(f)=f(x)\righttext{for}f\in \ag,x\in G.
\]

We will use the general fact that for any \csta\ $B$ the dual space $B^*$ can be regarded as the dual space of the multiplier algebra $M(B)$ when the latter is given the strict topology.
The \emph{zero set} of an ideal $I$ of \ag\ is 
$\{x\in G:f(x)=0\text{ for all }f\in I\}$, 
and
a \emph{spectral set} in $G$ is a closed subset $E$ for which $\{f\in \ag:f(x)=0\text{ for all }x\in E\}$
is the only closed ideal of \ag\ with zero set $E$.

Every coaction $(A,\delta)$ gives rise to a module structure on $A$ over the Fourier algebra $A(G)$ via slicing:
\[
f\cdot a:=\<\delta(a),\id\xt f\>\righttext{for}f\in A(G),a\in A.
\]
The same formula also makes $A$ a module over the Fourier-Stieltjes algebra $B(G)$.

If $H$ is a closed subgroup of $G$, then any coaction of $H$ on $A$ can be \emph{inflated} to a coaction of $G$ on $A$, using the canonical homomorphism of $\cst(H)$ into $M(\cstg)$ (see \cite[Example~A.29]{enchilada}).

We refer to \cite{fd} for \fb s\footnote{called \emph{\cst-algebraic bundles} in \cite{fd}} $\AA$ over $G$,
and to \cite{lprs} for the associated coaction $\delta_\AA$ --- often called the \emph{dual coaction} --- on the full cross-sectional algebra $\cst(\AA)$.
If $(A,\delta)$ is a coaction of $G$, then for $x\in G$ the \emph{spectral subspace} is
$A_x=\{a\in A:\delta(a)=a\xt x$,
and the \emph{spectrum} of $\delta$ is
$\spec=\{x\in G:A_x\ne \{0\}\}$.
If $\spec=G$ we say the coaction is \emph{effective}.
Dangerous bend: the spectrum is not always a subgroup of $G$
(see \cite[Example~5.3]{discrete} for a counterexample with $G=\Z_4$ --- however, note that  the \fb\ described in \cite{discrete} could alternatively be regarded as 
coming from an effective coaction of
$\Z_3$, as indicated in \exref{m2}).
When $G$ is finite (or, more generally, discrete,
or when
the fibres $A_x$ are only finitely nonzero),
the family $\{A_x\}_{x\in G}$ forms a Fell bundle over $G$,
and we say that $\delta$ is \emph{associated with $\AA$} if $(A,\delta)\simeq (\cst(\AA),\delta_\AA)$.
Every coaction of a discrete group has an associated Fell bundle, but
this is not automatic for arbitrary groups (see \cite[Example~2.3 (6)]{lprs} for a counterexample with $G=\T$).
When $G$ is finite the isomorphism classes of
\fb s $\AA$ and their associated coactions $\delta_\AA$ 
are in bijective correspondence (see \cite[Theorem~3.8 and Corollary~3.9]{discrete}).

The spectral subspace $A_e$ is also denoted by \fixed, and called the \emph{fixed-point subalgebra}.
A coaction $\delta$ is called \emph{ergodic} if $\fixed=\C 1$.
In this case, for every $x\in \spec$ the spectral subspace $A_x$ is spanned by a (nonunique) unitary $U_x$,
and the spectrum $\spec$ is a subgroup of $G$.

Suppose that $\delta$ is an effective ergodic coaction of $G$ on $A$.
Interestingly, for each $x\in G$ we get an automorphism $\ad U_x$ of $A$,
and in fact $x\mapsto \ad U_x$ is an action of $G$ on $A$.
When $A=\mn$
and $|G|=n^2$
this process can be reversed,
and we get a bijection between ergodic coactions and ergodic actions of $G$ on $A$.

We write $w_G$ for the unitary element of
\[
M(\cg\xt\cstg)=C(G,M^\beta(\cstg))\]
(where the ``$\beta$'' signifies that we mean continuity with respect to the strict topology on the multiplier algebra)
given by the canonical embedding $G\hookrightarrow M(\cstg)$.
If $\mu:\cg\to M(A)$ is a unital homomorphism, then the unitary element
\[
w_\mu:=(\mu\xt\id)(w_G)\in M(A\xt \cstg)
\]
satisfies
\[
(\id\xt\delta_G)(w_\mu)=(w_\mu\xt 1)(\id\xt\Sigma)(w_\mu),
\]
where $\Sigma$ is the ``flip automorphism'' of $\cstg\xt\cstg$ given on elementary tensors by $\Sigma(x\xt y)=y\xt x$,
and where \dg\ is the canonical coaction on \cstg\,
given on group elements by $\dg(x)=x\xt x$.
Unitaries satisfying this identity are called \emph{co\rep s}.
In the opposite direction,
\cite[Theorem~A.1]{naktak} proves that
every co\rep\ $w\in M(A\xt\cstg)$ is of the form $w_\mu$ for a unique \hm\ $\mu$ given by slicing:
\[
\mu(f)=\<w,\id\xt f\>\righttext{for}f\in \ag.
\]
The relevance for us is that for any $\mu$ there is an \emph{inner coaction} $\delta$ of $G$ on $A$ given by
\begin{equation}\label{implement}
\delta(a)=\ad w_\mu(a\xt 1)\righttext{for}a\in A.
\end{equation}
Such coactions are called \emph{unitary} in \cite[Examples~2.3 (3)]{lprs}
(and \cite[Lemma~1.11]{fullred} shows that conditions (b) and (c) in \cite{lprs} are redundant).
More generally, if $(A,\delta)$ is any coaction, we say that a unitary $w\in M(A\xt\cstg)$ \emph{implements} $\delta$,
and that $\delta$ is \emph{unitarily implemented},
if
\eqref{implement} holds.
Of course, not all unitaries will implement coactions,
and moreover a moment's thought reveals that a coaction that is implemented by a unitary $w$ might not be inner,
since slicing $w$ might not give a *-\hm\ of \ag\ (and hence of \cg).
For example, when $G$ is finite abelian we have ergodic actions, which are implemented by unitaries twisted by 2-cocycles.

We need the fact that if 
$c\in M(\cstg)$, then $\dg(c)=c\xt c$
implies that $c\in G$.
This is a folklore result concerning ``group-like elements'',
and is easy to verify by slicing and using
\[
\<\dg(c),f\xt g\>=\<c,fg\>\righttext{for}f,g\in \ag
\]
together with the identification of $G$ with the spectrum of \ag.

We frequently need to work with unitary multipliers of $M(\cstg)$, and for convenience we write
\[
UM(\cstg)=\{u\in M(\cstg):u\text{ is unitary}\}.
\]

In our study of effective ergodic coactions on \mn, we need to make contact with finite groups of \emph{central type}\footnote{We thank Erik B\'edos for informing us of this terminology and the references \cite{ginosar,schnabel}.}, 
defined in the literature as groups having a simple twisted group algebra; the motivation for the terminology has to do with central group extensions.
Groups of central type are of increasing interest in the past few decades.
The abelian groups of central type are quite well understood.
The smallest nonabelian examples have order 16,
where there are two nonisomorphic groups of central type.
In \exref{explicit}
we will discuss a couple of examples that we found in \cite{schnabel}.

We need a bit of the theory of 2-cocycles,
and we adopt the conventions of \cite{kleppner}:
Keep in mind that for any ergodic coaction $\delta$ of a finite group $G$ on a unital \csta\ $A$ each spectral subspace $A_x$ is spanned by a unitary, and any choice of a spanning set of unitaries $\{U_x\}_{x\in G}$,
where we always choose $U_e=1$,
gives a projective representation of $G$ in $A$, which by definition has
an associated 2-\emph{cocycle} $\omega\in Z^2(G,\T)$ determined by
\begin{equation}\label{cocycle}
U_xU_y=\omega(x,y)U_{xy}\righttext{for}x,y\in G,
\end{equation}
and which satisfies 
\begin{enumerate}
\item
$\omega(x,e)=\omega(e,x)=1$ for all $x\in G$
and

\item
$\omega(x,y)\omega(xy,z)=\omega(x,yz)\omega(y,z)$ for all $x,y,z\in G$.
\end{enumerate}
$U$ as above is called an \emph{$\omega$-\rep}.

As explained in
\cite[Propositions~VIII.10.10 and VIII.16.2]{fd},
$\omega$-\rep s of $G$,
\rep s of a cocycle bundle $\AA$ with cocycle $\omega$,
and central extensions of $\T$ by $G$
correspond up to isomorphism.
Choosing the spanning unitaries $U_x$
is almost the same as choosing the associated cocycle $\omega$ ---
the unitaries determine the cocycle uniquely,
but a different choice of unitaries could lead to the same cocycle.
Two cocycles $\omega$ and $\omega'$ are \emph{cohomologous} if 
there is a map $m:G\to\T$ such that
\[
\omega'(x,y)=m(x)m(y)\bar{m(xy)}\omega(x,y)\righttext{for}x,y\in G,
\]
and $\omega'$ is \emph{trivial} if it is cohomologous to $\omega=1$.
Every cocycle is cohomologous to a \emph{normalized} one, meaning that
\[
\omega(x,x\inv)=1\righttext{for all}x\in G,
\]
and we will always assume that our cocycle $\omega$ is normalized.
Choosing $\omega$ to be normalized is the same as choosing the spanning unitaries so that
\[
U_x^*=U_{x\inv}\righttext{for all}x\in G.
\]
A routine computation shows that for $\omega$ normalized we have
\[
\omega(x,y)\inv=\omega(y\inv,x\inv)\righttext{for all}x,y\in G.
\]
There will in general be many $\omega$-\rep s for any given cocycle $\omega$.
When $G$ is finite there will be $\omega$-\rep s $U$ that are
\emph{universal}
in the sense that for any $\omega$-\rep\ $V$ there is a unique \hm\
$\pi_V$ making the diagram
\[
\begin{tikzcd}
G \arrow[r,"U"] \arrow[dr,"V"']
&\cst(U(G)) \arrow[d,"\pi_V",dashed]
\\
&\cst(V(G))
\end{tikzcd}
\]
commute.
As usual, the pair $(\cst(U(G)),U)$ is unique up to isomorphism,
and we write
$\cstwg$ for any such algebra $\cst(U(G))$,
and call it the \emph{twisted group \csta} for $\omega$
(although we do not in general give a special notation to the associated $U$).
One choice of such a universal $\omega$-\rep\ is the 
\emph{right regular $\omega$-\rep}
\cite[page 556]{kleppner}
$R$ of $G$ on \ltwo,
defined by
\[
(R_x\xi)(y)=\omega(y,x)\xi(yx).
\]
We have
\[
\cstwg=\spn U(G)
\]
for any unversal $U$,
so
\[
\dim\cstwg=|G|.
\]
Also, any $\omega$-\rep\ $V$ is universal 
if and only if the associated \hm\ $\pi_V:\cstwg\to \cst(V(G))$ is faithful,
if and only if the unitaries $\{V_x:x\in G\}$ are linearly independent.

Later we will be interested in the possibility that \cstwg\ is simple, i.e., is isomorphic to \mn\ for some $n$. This obviously puts restrictions on $G$ --- for example, the order $|G|$ must be $n^2$. But this alone is insufficient.
It seems natural to say that $G$ \emph{has a simple twisted \csta} if
there exists a cocycle $\omega$ for which \cstwg\ is simple,
equivalently there exists an irreducible projective \rep\ $U$ of $G$
such that the unitaries $\{U_x\}_{x\in G}$ are linearly independent.
In the literature such a group is said to be of \emph{central type}.
There is a subtlety: the literature also contains numerous studies of 
\emph{faithful} irreducible projective \rep s $U$,
meaning that the composition of $U$ with the quotient map $\un\to \pun$
is faithful, equivalently the only $x\in G$ for which $U_x$ is a scalar operator is $x=e$.
This is weaker than the unitaries $\{U_x\}_{x\in G}$ being linearly independent.
For example,
it follows from \cite[Theorem~4.4]{ng} that
there are groups of order 16 with faithful irreducible projective \rep s $U$ on $\C^2$, and trivially $\{U_x\}_{x\in G}$ must be linearly dependent.
Interestingly, these groups are nonabelian.
This behavior cannot occur for abelian groups, because
Frucht \cite{frucht} proved that a finite abelian group has a faithful irreducible projective \rep\ if and only if it is isomorphic to $H\times H$ for some (abelian) group $H$,
and by \cite[paragraph following Corollary~1]{kleppner} such a group does have
a simple twisted \csta.
The smallest nonabelian finite groups of central type have order 16.

In \secref{s:unitary} we need to use the \emph{destabilization} process:
if $A$ 
is an elementary algebra (i.e., $A$ is isomorphic to the compact operators on a Hilbert space)
and 
$\iota:A\to M(B)$ is a nondegenerate homomorphism,
we define the \emph{relative commutant} of $A$ in $B$ as
\[
C(B,\iota):=\{c\in M(B):c\iota(a)=\iota(a)c\in B\text{ for all }a\in A\}.
\]
Then $B$ is a stabilization in a functorial way:
there is a unique isomorphism
\[
\theta:A\xt C(B,\iota)\variso B
\]
given on elementary tensors by $\theta(a\xt c)=\iota(a)c$.
We call this construction the destabilization process;
it is probably folkore, and is described explicitly in
\cite[Proposition~A.2]{gap},
\cite[Lemma~27.2]{exel},
\cite[Section~3]{fischer},
\cite[Theorem~2.1]{hjeror},
and
\cite[Proposition~3.4]{destabilization},
for example.

\section{\fb s}\label{s:fb}

We show that every coaction of a compact group on a finite-dimensional \csta\ $A$ is associated with a Fell bundle.

We will need to use the fact that every finite subset of $G$ is a spectral set for \ag;
this is implied by
\cite[Theorem~4]{warner},
since the empty set is a Ditkin set (a fact that is explicitly mentioned in the paragraph preceding Warner's Theorem 1).

\begin{thm}\label{fb}
Let $\delta$ be a coaction of a compact group $G$ on a finite-dimensional \csta\ $A$.
Then there is a \fb\ $\AA$ over $G$ such that $(A,\delta)$ is isomorphic to the dual coaction $(\cst(\AA),\delta_\AA)$.
\end{thm}

\begin{proof}
Let $a\in A$ be nonzero, and put
\[
(a)^\circ=\{\phi\in A(G):\phi\cdot a=0\}.
\]
Note that\[
\delta(a)=\sum_{i=1}^p e_i\xt c_i,
\]
where 
$\{e_1,\dots,e_p\}$ is a basis of $A$
and $c_i\in M(\cstg)$.
To ease the notational burden a bit, put $D=\{c_1,\dots,c_p\}$.
Because the $e_i$ are linearly independent it follows by slicing that $\phi\in (a)^\circ$
if and only if
$\phi$ is in the annihilator $D\ann$.

Now, $(a)^\circ$ is an ideal of \ag\ of finite codimension,
so is contained in only finitely many maximal ideals.
By \cite[Th\'eor\`em~3.34]{eymard}, the maximal ideals of \ag\ are of the form $I_x=\{\phi\in \ag:\phi(x)=0\}$, so the zero set of $(a)^\circ$ can be identified with a finite subset $E\subset G$.
As we mentioned before this proof,
$E$ is spectral
because it is finite (see \cite[Theorem~4]{warner}), consequently the ideal $(a)^\circ$ coincides with the intersection
\[
I_E=\bigcap_{x\in E}I_x=\{\phi\in \ag:\phi(x)=0\text{ for all }x\in E\}.
\]

Now, by general \cst-theory we can identify $\ag=(\cstg)^*$ with the dual space of $M(\cstg)$ when the latter is given the strict topology.
Then we can interpret the above as saying that the two subsets $D,E$ of $M(\cstg)$ have the same annihilator in \ag, i.e., 
\[
D\ann=E\ann.
\]
It then follows by the Bipolar Theorem that $D$ is contained in the double polar of $E$.
Since $E$ is finite, this double polar is just the linear span. Thus we have
\[
D\subset \spn E.
\]
Thus
\[
\delta(a)\in A\xt \spn E,
\]
and so there are $b_x\in A$ such that
\begin{equation}\label{delta a}
\delta(a)=\sum_{x\in E} b_x\xt x.
\end{equation}
We have
\begin{align*}
\sum_{x\in E} \delta(b_x)\xt x
&=(\delta\xt\id)\circ\delta(a)
\\&=(\id\xt\delta_G)\circ\delta(a)
\\&=\sum_{x\in E} b_x\xt x\xt x.
\end{align*}
Since 
$E$ is linearly independent in $M(\cstg)$,
we conclude that
\[
\delta(b_x)=b_x\xt x\righttext{for all}x\in E,
\]
i.e., $b_x$ is in the spectral subspace $A_{x}$.

Slicing \eqref{delta a} by $1\in A(G)$, we conclude that
$a=\sum_{x\in E} b_x$.
Letting $a$ run through a basis of $A$ we conclude that
$A$ is the sum of the spectral subspaces $\{A_x\}_{x\in\spec}$, where $\spec$ is the spectrum of the coaction $\delta$.
Now, the family $\AA:=\{A_{x}\}_{x\in G}$
is a Fell bundle, and since $A$ is finite-dimensional it must coincide with the full
cross-sectional algebra $\cst(\AA)$.
\end{proof}

\begin{rem}\label{general}
In the above theorem, we exploit the duality of the locally convex spaces
$M(\cst(G))$ and $A(G)$,
i.e. we are identifying $A(G)=B(G)$ with the dual space of $M(\cst(G))$ when the latter is given the strict topology.
In particular, we needed to know that $A(G)$ separates the points of $M(\cst(G))$.
More generally, this works as long as $G$ is amenable\footnote{Note that for nonamenable $G$ we have
$A(G)\subset B_r(G)$, the reduced Fourier algebra of $G$,
which kills the kernel of the regular representation of $G$).}.
In fact, \thmref{fb} is true for amenable $G$, and the current remark indicates that the above proof carries over.
\end{rem}

\begin{rem}\label{inflate}
In the proof of \thmref{fb} we only needed the spectral subspaces $A_x$ for $x$ in the finite subset $\spec$ of $G$.
This need not be a subgroup,
but when it is, the rest of $G$ is in some sense superfluous:
$\delta$ restricts to a coaction $\epsilon$ of $\spec$, and the given coaction of $G$ coincides with the inflation 
from $\spec$ to $G$,
as indicated in the commutative diagram
\[
\begin{tikzcd}
A \arrow[r,"\epsilon"] \arrow[dr,"\delta"']
&M(A\xt \cst(\spec)) \arrow[d,hook]
\\
&M(A\xt \cstg).
\end{tikzcd}
\]
\end{rem}

\section{Unitary implementation}\label{s:unitary}

We give an elementary proof of an implementation theorem
(proved for arbitrary Type~I factors in \cite[Theorem~3]{wassermann},
using more sophisticated machinery).

\begin{thm}\label{Mn implement}
Every coaction of a compact group on $M_n$ is unitarily implemented.
\end{thm}

\begin{proof}
Let $\delta:A\to M(A\xt\cstg)$ be a coaction where 
$A=\mn$.
Choose a decomposition of \cstg\ into a $c_0$-direct sum of matrix algebras:
\[
\cstg=\bigoplus_{i\in S}M_i.
\]
For each $i$ we have a central projection $p_i\in \cstg$ such that $p_i\cstg\simeq M_i$, and we fix such an isomorphism and in fact identify $p_i\cstg=M_i$.
Then we get
a \hm
\[
\delta_i=(1\xt p_i)\delta:A\to A\xt M_i.
\]

Since
$A$ is 
elementary,
we can apply the destabilization process:
write $C_i$ for the relative commutant of $\delta_i(A)$ in $M(A\xt M_i)$, so that
\[
A\xt M_i=\spn\{\delta(A)C_i\}\simeq A\xt C_i.
\]
Since $A=M_n$, we see that $C_i$ must also be a matrix algebra, and by counting
dimensions, we must have $C_i\simeq M_i$.
Fixing such an isomorphism, by composition we get an automorphism $\phi_i\in \aut(A\xt M_i)$ such that
\begin{align*}
\phi_i(A\xt 1)=\delta_i(A)\midtext{and}
\phi_i(1\xt M_i)=C_i.
\end{align*}
Since $A\xt M_i$ is elementary,
we can choose a unitary $U_i$ in $A\xt M_i$ such that
$\phi_i=\ad U_i$.
It follows that
\[
\delta_i(a)=\ad U_i(a\xt 1)\righttext{for all}a\in A.
\]

Now we put them together: define a unitary
\[
U=\sum_i U_i\in M(\cstg).
\]
Then we have
\[
\delta(a)=\sum_i\delta_i(a)=\sum_i \ad U_i(a\xt 1_{M_i})=\ad U(a\xt 1_{M(\cstg)}),
\]
so that $U$ implements $\delta$.
\end{proof}

\begin{lem}\label{different Us}
With the above notation, if $U$ is a unitary implementing $\delta$ then
the set of all unitaries implementing $\delta$ is
\[
\{U(1\xt u):u\in UM(\cstg)\}.
\]
\end{lem}

\begin{proof}
This follows quickly from doing it on each quotient by $1\xt p_i$,
then assembling them.
\end{proof}

\section{Inner}\label{s:inner}

In the preceding section we saw that all coactions of $G$ on \mn\ are implemented by unitaries $U\in M(\mn\xt\cstg)$.
The inner coactions of $G$ on \mn\ are precisely those
for which we can take
$U=(\mu\xt\id)(w_G)$
for a unital \hm\ $\mu:C(G)\to \mn$.

We pause to see what $\mu$ must look like.
By the Spectral Theorem we can choose
orthogonal minimal
projections $p_1,\dots,p_n$
and $x_1,\dots,x_n\in G$ such that
\[
\mu(f)=\sum_i f(x_i)p_i.
\]
Moreover, any such choice of $p_i$ and $x_i$ will give a \hm\ $\mu:\cg\to \mn$.

For any coaction $\delta$ of $G$ on \mn, observe that $\fixed$ is a unital \cst-subalgebra of $M_n$;
we get some information by
looking at its maximal abelian \cst-subalgebras.

\begin{thm}\label{fixed inner}
Let $\delta$ be a coaction of $G$ on $A=\mn$ with $n>1$.
Then $\delta$ is inner if and only if
its fixed-point algebra $\mn^\delta$ has an abelian \cst-subalgebra of dimension $n$.
\end{thm}

\begin{proof}
First suppose that 
$M^\delta_n$ has
an abelian \cst-subalgebra of dimension $n$.
Then we can find rank-one projections $p_1,\dots p_n\in \fixed$ such that $\sum_ip_i=1$.
Choose matrix units $e_{ij}$ such that $e_{ii}=p_i$ for $i=1,\dots,n$.
By \thmref{Mn implement} we can choose a unitary $U$ implementing $\delta$.
We can 
write
\[
U=\sum_{i,j} e_{ij}\xt y_{ij}\righttext{with}y_{ij}\in M(\cstg).
\]
For each $k$ we have
\begin{align*}
p_k\xt 1
&=\delta(p_k)
\\&=\sum_{i,r,j,s}e_{ir}p_k e_{sj}\xt y_{ir}y_{js}^*
\\&=\sum_{i,j}e_{ij}\xt y_{ik}y_{jk}^*.
\end{align*}
Because the $e_{ij}$ are linearly independent,
all terms except when $i=j=k$
are zero,
and 
$y_{kk}y_{kk}^*=1$, so $y_{kk}$ is a unitary.
On the other hand, with $i\ne k$ and $j=k$ we get
\[
0=y_{ik}y_{kk}^*,
\]
so $y_{ik}=0$ since $y_{kk}\in UM(\cstg)$.
Thus
\[
U=\sum_i p_i\xt y_i,
\]
where we have written $y_i=y_{ii}$ to
simplify the notation.
For 
all $i,j$
we have
\[
\delta(e_{ij})=\sum_{r,s} p_re_{ij}p_s\xt y_ry_s^*=e_{ij}\xt y_iy_j^*.
\]
The coaction identity 
implies that
$\delta_G(y_iy_j^*)=y_iy_j^*\xt y_iy_j^*$, so $y_iy_j^*\in G$ and $e_{ij}\in A_{y_iy_j^*}$.

Now we adjust $U$ so that $y_i$ itself is in $G$:
recall that if $U'$ is any unitary implementing $\delta$ then $U'=U(1\xt u)$ for some $u\in UM(\cstg)$,
and moreover every choice of $u\in UM(\cstg)$ gives a suitable $U'$.
If we take $u=y_n^*$,
then
$x_i:=y_iy_n^*\in G$ for $i=1,\dots,n$
(and $x_n=e$),
and therefore we now can
replace the original $U$ by
\[
U=\sum_i p_i\xt x_i
\]
with $x_i\in G$.

We can define a \hm\ $\mu:\cg\to \mn$ by
\[
\mu(f)=\sum_i p_if(x_i),
\]
and then
\[
U=(\mu\xt\id)(w_G),
\]
so $\delta$ is inner.

Conversely, suppose that $\delta$ is inner, determined by a unital \hm\ $\mu:\cg\to A$.
Choose
mutually orthogonal minimal projections $p_1,\dots,p_n\in \mn$
and $x_1,\dots,x_n\in G$ such that
\[
\mu(f)=\sum_i f(x_i)p_i.
\]
Thus $\delta$ is implemented by
\[
(\mu\xt\id)(w_G)=\sum_i p_i\xt x_i,
\]
and
a quick computation shows that
\[
\delta(a)=\sum_{i,j} p_iap_j\xt x_ix_j\inv.
\]
Choosing matrix units $e_{ij}$ such that $e_{ii}=p_i$, we have $e_{ij}\in A_{x_ix_j\inv}$.
In particular, $p_i=e_{ii}\in A_e=A^\delta$ for all $i$, so
the fixed-point algebra $A^\delta$ has an abelian \cst-subalgebra of dimension $n$.
\end{proof}

Paraphrasing the above theorem and its proof, we see that an inner coaction $\delta$ on $A=M_n$ is implemented by
\[
U=\sum_1^np_i\xt x_i,
\]
where $p_i=e_{ii}$ and $x_i\in G$.
For the matrix units we have
\[
e_{ij}\in A_{x_ix_j\inv}.
\]
Moreover, the $x_i$'s are only unique up to multiplying (on the right) by any fixed $x\in G$.
Here are some easy consequences:

\begin{prop}\label{eij}
With the above notation, we have:
\begin{enumerate}[label=\textup{(\arabic*)}]
\item
$e_{ij}\in A^\delta$ if and only if $x_i=x_j$;

\item
For any particular pair $i,j$, if $e_{ij}\in A^\delta$ then without loss of generality we can modify the $x$'s so that $x_i=x_j=e$,
and then for all $k\ne i,j$ we have
\begin{align*}
e_{ik}&\in A_{x_k\inv}\\
e_{kj}&\in A_{x_k}.
\end{align*}
\end{enumerate}
\end{prop}

\begin{ex}\label{m2}
For inner coactions $\delta$ of $G$ on $A=M_2$,
since $e_{22}\in A^\delta$
the only choice is the element $x\in G$ for which $e_{12}\in A_x$,
and then the spectrum is $\spec=\{e,x,x\inv\}$.
If $x=x\inv$ then $\spec=\Z_2$.
In general $\spec$ may not be a subgroup of $G$,
unless $x^2=x\inv$, in which case $\spec=\Z_3$.
Note that if $x^2\ne x\inv$ we could still ``regard'' this as coming from an effective
coaction of $\Z_3$,
but this is of questionable value since $G$ might not have $\Z_3$ as a subgroup.

Of course, if $x=e$ then $\delta$ is trivial and $A^\delta=A$.
\end{ex}

\begin{ex}\label{m3}
For inner coactions $\delta$ of $G$ on $A=M_3$,
if for example
\[
A^\delta=\spn\{e_{ij}:1\le i,j\le 2\text{ or }i=j=3\},
\]
then $\delta$ is implemented by a co\rep\ of the form
\[
U=e_{11}\xt 1+e_{22}\xt 1+e_{33}\xt x
\]
for some $x\in G$.
This is a routine application of \propref{eij} and the discussion immediately preceding it.
\end{ex}

\section{Ergodic}\label{s:ergodic}

Focusing on coactions of our compact group $G$ on $M_n$, there are some immediate, easy facts.

\begin{lem}\label{trivial}
If $G$ has
an effective ergodic coaction on \mn\ 
with $n>1$,
then:
\begin{enumerate}[label=\textup{(\arabic*)}]
\item
$G$ has order $|G|=n^2$.

\item
If $G$ is nonabelian then $n>2$.

\item
$G$ is noncyclic.
\end{enumerate}
\end{lem}

\begin{proof}
(1)
If $\delta$ is an ergodic coaction of $G$ on $A=\mn$,
then (see \cite{katson})
every spectral subspace $A_x$ is spanned by a unitary, and
the subspaces $\{A_x:x\in G\}$ are linearly independent,
so $|G|=n^2$.

(2)
This follows immediately from (1).

(3)
This follows quickly from the triviality of the 2-cohomology $H^2(G,\T)$ when $G$ is cyclic; see \cite[Proposition~III.5.11]{fd1} for example.
But here is a quick and dirty argument:
Suppose $G$ is cyclic with generator $x$, and pick a unitary $u\in A_x$.
Then $\{u^i:i=1,\dots,n\}$ is a choice of unitaries in the spectral subspaces,
from which it follows that \mn\ is commutative, which is absurd.
\end{proof}

\lemref{trivial} is almost trivial.
The (seemingly) slightly stronger
\propref{xy} below is easy but relies on somewhat more sophisticated technology.

\begin{prop}\label{xy}
If $G$ has an ergodic coaction on \mn\ with $n>1$,
then $G$ has a finite abelian noncyclic subgroup.
\end{prop}

\begin{proof}
Suppose that $\delta$ is an ergodic coaction on $A=\mn$.
Then the spectrum $\spec$ is a subgroup of order $n^2$,
and without loss of generality we suppose that $G=\spec$.
Choose spanning unitaries $U_x$ for the spectral subspaces $A_x$,
and let $\omega$ be the associated 2-cocycle.
Then $U$ is 
an irreducible $\omega$-\rep\ of $G$.
In \cite[page~559]{kleppner} Kleppner defines $x\in G$ to be \emph{$\omega$-regular} if
\[
\omega(x,y)=\omega(y,x)\righttext{whenever}xy=yx.
\]
and it follows from
\cite[Corollary~1]{kleppner} that
$G\minus \{e\}$ contains no $\omega$-regular elements.
Thus, picking any $x\ne e$ we can find $y\ne e$ such that $xy=yx$ and $\omega(x,y)\ne \omega(y,x)$.
Let $H=\<x,y\>$ be the abelian subgroup generated by $x$ and $y$, so $\omega$ restricted to $H\times H$ is a non-trivial cocycle.
As we mentioned in the preceding proof,
$H$ must be noncyclic.
\end{proof}

\begin{rem}
When $G$ is abelian, Kleppner's characterization mentioned in the above proof
can actually be reformulated as 
\cite[Theorem~5.9]{optergodic}.
Even if $G$ is not assumed to be abelian, it is possible to give a characterization generalizing that in \cite{optergodic}.
\end{rem}

\begin{prop}\label{prime}
If $G$ has 
an irreducible projective \rep\ $U$ on $\C^p$
with $p$ prime
and $\{U_x:x\in G\}$ linearly independent, then
$G=\Z_p\times \Z_p$.With $\gamma=\exp(2\pi i/p)$, then up to equivalence
\[
U(1,0)=\diag(\gamma^j  \mid j=1\cdots p)  \midtext{and} 
U(0,1)=\{\delta(j,j+1)\mid j\in \Z_p\}.
\]
\end{prop}

\begin{proof}
From the assumptions there are $x,y\in G$ with $yx=xy$ and 
\[
\omega(x,y)\ne \omega(y,x).
\]
Since 
$|G|=p^2$,
the order of $x$ and $y$ must be  $1$,   $p$  or $p^2$ (remember that we assume $G$ is finite).
One quickly rules out $1$ and $p^2$. Then one realizes that one can pick  $x,y$ such that 
\[
\omega(x,y)\bar{\omega(y,x)}=\exp(2\pi i/p)=:\gamma.
\]
We must have $H\cap K=\{e\}$, so
$|HK|= |H| \, |K|=p^2= |G|$, and hence $G=\Z_p\times \Z_p$.
We leave as an exercise to show that $U$ as described is (up to equivalence) the only irreducible projective representation of $\Z_p\times \Z_p$ satisfying the theorem.
\end{proof}

\begin{ex}\label{su}
It follows from \propref{xy} that $\su$ has no ergodic coactions on any \mn\ with $n>1$,
because every finite abelian subgroup is cyclic (folklore).
Interestingly, the quotient group $\so$ has an ergodic coaction on $M_2$,
because it contains the Klein-4 group $\Z_2\times \Z_2$ as a subgroup;
for example, the elements
\[
\mtx{1&0&0\\0&-1&0\\0&0&-1}
\midtext{and}
\mtx{-1&0&0\\0&-1&0\\0&0&1}
\]
are suitable generators.
\end{ex}

\begin{rem}
\cite[Discussion following Corollary~1]{kleppner}
shows that if $H$ is any abelian group of order $n$,
then there is a 2-cocycle $\omega$ on $G=H\times \what H$ given by
\[
\omega\bigl((x,\chi),(y,\sigma)\bigr)=\chi(y),
\]
and no element of $G\minus\{e\}$ is $\omega$-regular,
so $G$
is of central type
and hence has
an effective ergodic coaction on \mn.
Note that the set-up in \propref{prime} is a special case.
In fact, Frucht proved in \cite{frucht} that a finite abelian group has a
faithful
irreducible projective \rep\
if and only if it is isomorphic to the square of an abelian group,
in which case it
is of central type,
as we explained in the the preliminaries.
This was rediscovered somewhat later by Olesen, Pedersen, and Takesaki \cite[Theorem~5.9]{optergodic}.

Ng \cite{ng} generalized this to nonabelian groups;
his setting was primarily metacyclic groups (extensions of cyclic by cyclic).
\cite[Proposition~4.1]{ng} shows that if $G$ is such a group, with generators $a$ and $b$,
if there is a faithful irreducible projective \rep\
then by \cite[Theorem~5.2]{ng} $G$ is the semidirect product of $\<a\>$ and $\<b\>$.
Moreover, if $|G|$ is a power of an odd prime,
then $G$ has a faithful irreducible projective \rep\ if and only if the subgroups $\<a\>$ and $\<b\>$ are isomorphic.
So, for $G$ to be nonabelian, the smallest such example would have order $81=3^4$.
\end{rem}

\begin{ex}\label{explicit}
By \cite[Theorems~3.8 and 3.9]{schnabel},
there are two nonabelian groups of central type of order 16 (the smallest possible).
Here
they are, with
examples of irreducible projective representations of degree 4:

\begin{itemize}
\item
$G=\bigl\<a,b,c\bigm| a^4=b^2=c^2=e,ab=ba,bc=cb,cac\inv=a\inv\bigr\>$.
For a suitable projective representation take
\[
A=\mtx{1&0&0&0\\0&-1&0&0\\0&0&i&0\\0&0&0&-i},
\quad
B=\mtx{0&1&0&0\\1&0&0&0\\0&0&0&1\\0&0&1&0},
\quad
C=\mtx{0&0&1&0\\0&0&0&1\\1&0&0&0\\0&1&0&0}
\]
Then
\[
AB=-BA,
\quad
A^4=B^2=C^2=I,
\quad
BC=CB,
\quad
CAC\inv=iA\inv.
\]

\item
$G=\bigl\<a,b,c\bigm| a^4=b^2=c^2=e,ab=ba,bc=cb,cac\inv=ab\bigr\>$.
Take the same $A,B$ as above, but now take
\[
C=\frac 1{\sqrt 2}\mtx{0&0&-i&1\\0&0&1&-i\\i&1&0&0\\1&i&0&0}.
\]
The only new relation is
$CAC\inv=AB$.
\end{itemize}

It seems there are rather few finite nonabelian groups of central type.
For example, in 
\cite[Theorem~3.21]{ginosar}
it is shown that if $G$ is nonabelian of central type and $|G|=p^2q^2$
with primes $p<q$,
then the number of nonisomorphic possibilities is one if $pq\ne 6$ and two if $pq=6$.
\end{ex}

\section{$G=S_3$}\label{s:s3}

It is interesting to examine coactions $\delta$ of the smallest nonabelian group $S_3$ (the permutations of a set of three elements) on $A=M_n$.
We fix a presentation
\[
S_3=\<\alpha,\beta\mid\alpha^3=\beta^2=e,\alpha\beta=\beta\alpha^2\>.
\]
Since all proper subgroups of $S_3$ are cyclic, to avoid trivialities we look for effective coactions.
By \lemref{trivial} this requires $n\ge 3$,
and we cannot expect $\delta$ to be ergodic.

Here's where we can make some progress: look for effective \emph{inner} coactions.
As we have observed, we need $n\ge 3$.
But even $n=3$ will not work:
in the notation of the proof of \thmref{fixed inner},
we have $\spec=\{x_ix_j\inv\}$ with $x_i\in S_3$, and we can assume $x_3=1$.
So, we can choose $x_1, x_2$, then
we have $e_{12}\in A_{x_1x_2\inv}$.
Since $S_3$ has three self-inverse elements, trivial computations show that we can have at most 
three non-identity elements 
in $\spec$, which is absurd.
Therefore we cannot have an effective inner coaction of $S_3$ on $M_3$.

Therefore, an effective inner coaction of $S_3$ on \mn\ requires $n\ge 4$.
And then it is easy to generate examples:
as we mentioned above, we might as well choose $x_4=1$.
Then, for example, we can choose
\[
x_1=\alpha,
\quad x_2=\beta,
\midtext{and}x_3=\alpha\beta.
\]
Then $\{x_ix_j\inv:i,j=1,\dots,3\}$ will contain
\[
x_1x_2\inv=\alpha\beta\midtext{and}
x_2x_3\inv=\beta\alpha=\alpha^2\beta,
\]
and we have accounted for all six elements of $S_3$.
Trivially we could have made other choices, and we could do something similar for any $n\ge 5$ as well.
Conclusion:

\begin{prop}\label{s3 inner}
For any $n\ge 4$ the above procedure gives all effective inner coactions of $S_3$ on \mn,
but there are none for $n<4$.
\end{prop}

\section{Conclusion}\label{sec:end}

Clearly, we have only scratched the surface here;
there are many aspects of compact coactions on \mn\ that need more investigation.
For example, \secref{s:inner} leads to the following train of thought:
the 
center of the
fixed-point algebra $A^\delta$
is spanned by
projections $\{p_i\}_{i\in S}$,
and the coaction $\delta$ restricts to the corners $p_iAp_i$.
But a bit more thought reveals that we have a ``coaction'' on a Fell bundle $\{p_iAp_j\}_{i,j\in S}$ over the groupoid 
$S\times S$ with multiplication $(i,j)(j,k)=(i,k)$.
It is tantalizing to think that this observation could yield significant further insight into the coaction $\delta$.

Also, coactions of compact groups have been recently shown to be \emph{Pedersen rigid} in the sense that the outer conjugacy class is completely determined by the isomorphism class of the dual action (see \cite{compactrigid}). How is this related to our results in this paper?

A next step in the study of coactions of the compact group $G$ is to ``ramp up'' from \mn\ to $\KK$ (the compact operators on a separable infinite-dimensional Hilbert space).
Perhaps one should even investigate more general AF-algebras in this connection?
In particular, how about coactions of $G$ on UHF algebras, e.g., $M_n\xt M_n\xt\cdots$,
investigated by Masuda in \cite{masuda}?

If $U$ is a unitary implementing the coaction $\delta$ as in \thmref{Mn implement}, it is possible to show that there is a unique unitary $u\in M(\cstg\xt\cstg)$ such that
\[
(\id\xt\delta_G)(U)=U_{12}U_{13}(1\xt u).
\]
This looks like some sort of ``2-cocycle'' for $U$,
and perhaps it would be fruitful to regard $U$ as a sort of ``twisted co\rep''?
This would perhaps lead to a cohomology theory of coactions.
In fact, this is explored in \cite{landstad,wassermann}.
The ultimate goal here would be to classify
coactions of a compact group $G$ on matrix algebras \mn.

%\bibliography{cptcoactions.bib}

\begin{thebibliography}{KOQT23}

\bibitem[BKQT23]{gap}
Erik B\'{e}dos, S.~Kaliszewski, John Quigg, and Jonathan Turk.
\newblock Coactions on {$C^*$}-algebras and universal properties.
\newblock {\em Appl. Categ. Structures}, 31(5):14 pp, 2023.

\bibitem[EKQR06]{enchilada}
Siegfried Echterhoff, S.~Kaliszewski, John Quigg, and Iain Raeburn.
\newblock A categorical approach to imprimitivity theorems for
  {$C^*$}-dynamical systems.
\newblock {\em Mem. Amer. Math. Soc.}, 180(850):viii+169, 2006.

\bibitem[Exe17]{exel}
Ruy Exel.
\newblock {\em Partial dynamical systems, {F}ell bundles and applications},
  volume 224 of {\em Mathematical Surveys and Monographs}.
\newblock American Mathematical Society, Providence, RI, 2017.

\bibitem[Eym64]{eymard}
Pierre Eymard.
\newblock L'alg\`ebre de {F}ourier d'un groupe localement compact.
\newblock {\em Bull. Soc. Math. France}, 92:181--236, 1964.

\bibitem[FD88a]{fd1}
J.~M.~G. Fell and R.~S. Doran.
\newblock {\em Representations of {$^*$}-algebras, locally compact groups, and
  {B}anach {$^*$}-algebraic bundles. {V}ol. 1}, volume 125 of {\em Pure and
  Applied Mathematics}.
\newblock Academic Press, Inc., Boston, MA, 1988.
\newblock Basic representation theory of groups and algebras.

\bibitem[FD88b]{fd}
J.~M.~G. Fell and R.~S. Doran.
\newblock {\em Representations of {$^*$}-algebras, locally compact groups, and
  {B}anach {$^*$}-algebraic bundles. {V}ol. 2}, volume 126 of {\em Pure and
  Applied Mathematics}.
\newblock Academic Press, Inc., Boston, MA, 1988.
\newblock Banach $^*$-algebraic bundles, induced representations, and the
  generalized Mackey analysis.

\bibitem[Fis04]{fischer}
R.~Fischer.
\newblock Maximal coactions of quantum groups.
\newblock SFB 478 -- Geometrische Strukturen in der Mathematik, 350, Univ.
  M\"unster, 2004.

\bibitem[Fru32]{frucht}
Robert Frucht.
\newblock \"{U}ber die {D}arstellung endlicher {A}belscher {G}ruppen durch
  {K}ollineationen.
\newblock {\em J. Reine Angew. Math.}, 166:16--29, 1932.

\bibitem[GS19]{ginosar}
Yuval Ginosar and Ofir Schnabel.
\newblock Groups of central-type, maximal connected gradings and intrinsic
  fundamental groups of complex semisimple algebras.
\newblock {\em Trans. Amer. Math. Soc.}, 371(9):6125--6168, 2019.

\bibitem[HRr98]{hjeror}
Jacob v.~B. Hjelmborg and Mikael R\o rdam.
\newblock On stability of {$C^*$}-algebras.
\newblock {\em J. Funct. Anal.}, 155(1):153--170, 1998.

\bibitem[KOQ16]{destabilization}
S.~Kaliszewski, Tron Omland, and John Quigg.
\newblock Destabilization.
\newblock {\em Expo. Math.}, 34(1):62--81, 2016.

\bibitem[KOQT23]{compactrigid}
S.~Kaliszewski, Tron Omland, John Quigg, and Jonathan Turk.
\newblock Strong {P}edersen rigidity for coactions of compact groups.
\newblock {\em Internat. J. Math.}, 34(13):Paper No. 2350083, 13, 2023.

\bibitem[KS82]{katson}
Yoshikazu Katayama and Gu~Song.
\newblock Ergodic co-actions of discrete groups.
\newblock {\em Math. Japon.}, 27(2):159--175, 1982.

\bibitem[Kle62]{kleppner}
Adam Kleppner.
\newblock The structure of some induced representations.
\newblock {\em Duke Math. J.}, 29:555--572, 1962.

\bibitem[Lan92]{landstad}
Magnus~B. Landstad.
\newblock Ergodic actions of nonabelian compact groups.
\newblock In {\em Ideas and methods in mathematical analysis, stochastics, and
  applications ({O}slo, 1988)}, pages 365--388. Cambridge Univ. Press,
  Cambridge, 1992.

\bibitem[LPRS87]{lprs}
M.~B. Landstad, J.~Phillips, I.~Raeburn, and C.~E. Sutherland.
\newblock Representations of crossed products by coactions and principal
  bundles.
\newblock {\em Trans. Amer. Math. Soc.}, 299(2):747--784, 1987.

\bibitem[Mas08]{masuda}
Toshihiko Masuda.
\newblock Classification of actions of duals of finite groups on the {AFD}
  factor of type {${\rm II}_1$}.
\newblock {\em J. Operator Theory}, 60(2):273--300, 2008.

\bibitem[NT79]{naktak}
Yoshiomi Nakagami and Masamichi Takesaki.
\newblock {\em Duality for crossed products of von {N}eumann algebras}, volume
  731 of {\em Lecture Notes in Mathematics}.
\newblock Springer, Berlin, 1979.

\bibitem[Ng76]{ng}
H.~N. Ng.
\newblock Faithful irreducible projective representations of metabelian groups.
\newblock {\em J. Algebra}, 38(1):8--28, 1976.

\bibitem[OPT80]{optergodic}
Dorte Olesen, Gert~K. Pedersen, and Masamichi Takesaki.
\newblock Ergodic actions of compact abelian groups.
\newblock {\em J. Operator Theory}, 3(2):237--269, 1980.

\bibitem[Qui94]{fullred}
John~C. Quigg.
\newblock Full and reduced {$C^*$}-coactions.
\newblock {\em Math. Proc. Cambridge Philos. Soc.}, 116(3):435--450, 1994.

\bibitem[Qui96]{discrete}
John~C. Quigg.
\newblock Discrete {$C^*$}-coactions and {$C^*$}-algebraic bundles.
\newblock {\em J. Austral. Math. Soc. Ser. A}, 60(2):204--221, 1996.

\bibitem[Sch16]{schnabel}
Ofir Schnabel.
\newblock Simple twisted group algebras of dimension {$p^4$} and their
  semi-centers.
\newblock {\em Comm. Algebra}, 44(12):5395--5425, 2016.

\bibitem[War76]{warner}
C.~Robert Warner.
\newblock A class of spectral sets.
\newblock {\em Proc. Amer. Math. Soc.}, 57(1):99--102, 1976.

\bibitem[Was88]{wassermann}
Antony Wassermann.
\newblock Ergodic actions of compact groups on operator algebras. {II}.
  {C}lassification of full multiplicity ergodic actions.
\newblock {\em Canad. J. Math.}, 40(6):1482--1527, 1988.

\end{thebibliography}
%\bibliographystyle{alpha}

\end{document}